\documentclass[12pt, reqno]{amsart}
\usepackage{ amsmath,amsthm, amscd, amsfonts, amssymb, graphicx, color}
\usepackage[bookmarksnumbered, colorlinks, plainpages]{hyperref}
\newtheorem{thm}{Theorem}[section]
\newtheorem{cor}[thm]{Corollary}
\newtheorem{lem}[thm]{Lemma}
\newtheorem{prop}[thm]{Proposition}

\newtheorem{exam}[thm]{Example}
\numberwithin{equation}{section}

\begin{document}

\title{Generalized Hirano Inverses in Banach algebras}

\author{Huanyin Chen}
\author{Marjan Sheibani$^*$}
\address{
Department of Mathematics\\ Hangzhou Normal University\\ Hang -zhou, China}
\email{<huanyinchen@aliyun.com>}
\address{Women's University of Semnan (Farzanegan), Semnan, Iran}
\email{<sheibani@fgusem.ac.ir>}

 \thanks{$^*$Corresponding author}


\begin{abstract}
Let $A$ be a Banach algebra. An element $a\in A$ has generalized Hirano inverse if there exists $b\in A$ such that $$b=bab, ab=ba, a^2-ab\in A^{qnil}.$$ We prove that $a\in A$ has generalized Hirano inverse if and only if $a$ has g-Drazin inverse and $a-a^3\in A^{qnil}$, if and only if there exists $p^3=p\in comm(a)$ such that $a-p\in A^{qnil}$. The Cline's formula for generalized Hirano inverses are thereby obtained. Let $a,b\in A$ have generalized Hirano inverse. If $a^2b=aba$ and $b^2a=bab$, we prove that $a+b$ has generalized Hirano inverse if and only if $1+a^db$ has generalized Hirano inverse.
Hirano inverses of operator matrices over Banach spaces are also studied.\end{abstract}

\maketitle

\section{Introduction}

Let $A$ be a Banach algebra with an identity. The commutant of $a\in A$ is defined by $comm(a)=\{x\in
R~|~xa=ax\}$. The double commutant of $a\in A$ is defined by $comm^2(a)=\{x\in A~|~xy=yx~\mbox{for all}~y\in comm(a)\}$. An element $a\in A$ has g-Drazin inverse (i.e., generalized Drazin inverse) in case there exists $b\in A$ such that $$b=bab, b\in comm(a), a-a^2b\in A^{qnil}.$$ The preceding $b$ is unique if exists, we denote it by $a^d$. Here, $A^{qnil}$ denote the set of all quasinilpotents of the ring $R$, i.e., $$A^{qnil}=\{ a\in A~|~1+ax\in A~\mbox{invertible for all}~x\in comm(a)\}.$$

The motivation of this paper is to extend such generalized inverses in Banach algebras to a wider case by means of tripotents $p$, i.e, $p^3=p$. An element $a\in A$ has generalized Hirano inverse if there exists $b\in A$ such that $$b=bab, b\in comm(a), a^2-ab\in A^{qnil}.$$ We may replace the double commutator for the commutator in the preceding definition for a Banach algebra (see \cite{Mo}). Many elementary propertes of Hirano inverses were investiagted in [CS].

As is well known, $a\in A$ has g-Drazin inverse if and only if there exists an idempotent $e\in comm(a)$ such that $a+e\in A$ is invertible and $ae\in A^{qnil}$. Here, the spectral idempotent $e$ is unique, and denote it by $a^{\pi}$. In Section 2, we prove that $a\in A$ has generalized Hirano inverse if and only if $a$ has g-Drazin inverse and $a-a^3\in A^{qnil}$, if and only if there exists $p^3=p\in comm(a)$ such that $a-p\in A^{qnil}$.

Let $a,b\in R$. Then $ab$ has g-Drazin inverse if and only if $ba$ has g-Drazin inverse and
$(ba)^{d} = b((ab)^{d})^2a$. This was known as Cline's formula for g-Drazin inverses (see \cite{L}).
In Section 3, we extend Cline's formula for generalized Hirano inverses.

In Section 4, we are concern on additive property for generalized Hirano inverses. Let $a,b\in A$ have generalized Hirano inverses. If $a^2b=aba$ and $b^2a=bab$, we prove that $a+b$ has generalized Hirano inverse if and only if $1+a^db$ has generalized Hirano inverse.

Finally, in the last section, we investigate generalized Hirano inverses for operator matrices over Banach spaces.

Throughout the paper, all Banach algebra are complex with identity $1$. We use $U(A)$ to denote the set of all units in $A$. ${\Bbb N}$ stands for the set of all natural numbers.

\section{generalized Hirano inverses}

The aim of this section is to present new characterizations of generalized Hirano inverses which will be used repeatedly. We begin with

\begin{lem} \cite[Lemma 2.10 and Lemma 2.11]{ZMC} Let $A$ be a Banach algebra, $a,b\in A$ and $a^2b=aba, b^2a=bab$.\end{lem}
\begin{enumerate}
\item [(1)]{\it If $a,b\in A^{qnil}$, then $a+b\in A^{qnil}$.}
\vspace{-.5mm}
\item [(2)]{\it If $a$ or $b\in A^{qnil}$, then $ab\in A^{qnil}$.}
\end{enumerate}

\begin{lem} Let $A$ be a Banach algebra, and let $a\in A$. Then $a$ has gs-Drazin inverse if and only if\end{lem}
\begin{enumerate}
\item [(1)]{\it $a$ has g-Drazin inverse;}
\vspace{-.5mm}
\item [(2)]{\it $a-a^2\in A^{qnil}$.}
\end{enumerate}\begin{proof} $\Longrightarrow$ Write $a=e+w$ with $e^2=e\in comm^2(a),w\in A^{qnil}$. Then $a+(1-e)=1+w\in U(R)$ and $a(1-e)=(1-e)w\in A^{qnil}$.
Therefore $a$ has g-Drazin inverse. Moreover, $a-a^2=(1-2e-w)w\in A^{qnil}$, by Lemma 2.1.

$\Longleftarrow$ Since $a$ has g-Drazin inverse, we can find som $e^2=e\in comm(a)$ such that $u:=a+e\in U(A)$ and $ae\in A^{qnil}$.
Then $v:=a-e=a(1-e)+(ae-e)=u(1-e)+(ae-1)e\in U(R)$. Hence, $a-a^2=(e+v)-(e+v)^2=-v(2e+v-1)$. This shows that $a-(1-e)=-v^{-1}(a-a^2)$. In light of Lemma 2.1, $a-(1-e)\in A^{qnil}$. This completes the proof.\end{proof}

\begin{lem} Let $A$ be a Banach algebra, and let $a\in A$. Then the following are equivalent:\end{lem}
\begin{enumerate}
\item [(1)]{\it $a$ has generalized Hirano inverse.}
\vspace{-.5mm}
\item [(2)]{\it $a^2\in A$ has gs-Drazin inverse.}
\vspace{-.5mm}
\item [(3)]{\it There exists $b\in comm(a)$ such that $$b=(ab)^2, a^2-a^2b\in A^{qnil}.$$}
\end{enumerate}
\begin{proof} $(1)\Rightarrow (3)$ By hypothesis, there exists $c\in comm(a)$ such that $c=c^2a$ and $a^2-ac\in A^{qnil}$. Let $b=c^2$. Then $b\in comm(a), b=c^4a^2=b^2a^2=(ab)^2$. Moreover, we have $a^2-a^2b=a^2-ac\in A^{qnil}$, as desired.

$(3)\Rightarrow (2)$ By hypothesis, there exists $b\in comm(a)$ such that $b=(ab)^2, a^2-a^2b\in A^{qnil}.$ Hence $b\in comm(a^2), b=ba^2b$. Therefore $a^2\in A$ has gs-Drazin inverse.

$(2)\Rightarrow (1)$ Since $a^2\in A$ has gs-Drazin inverse, then there exists $c\in comm^2(a^2)$ such that $c=c^2a^2$ and $a^2-a^2c\in A^{qnil}$.
Set $b=ac$. Then $ab=ba, b=b^2a$ and $a^2-ab=a^2-a^2c\in A^{qnil}$. Therefore $a$ has generalized Hirano inverse, as asserted.\end{proof}

\begin{thm} Let $A$ be a Banach algebra, and let $a\in A$. Then $a$ has generalized Hirano inverse if and only if\end{thm}
\begin{enumerate}
\item [(1)]{\it $a$ has g-Drazin inverse;}
\vspace{-.5mm}
\item [(2)]{\it $a-a^3\in A^{qnil}$.}
\end{enumerate}\begin{proof} $\Longrightarrow$. In view of Lemma 2.3, $a^2\in A$ has gs-Drazin inverse. It follows by Lemma 2.2 and \cite[Theorem 2.7]{Y}
Furthermore, $a(a-a^3)=a^2-a^4\in A^{qnil}$, and so $(a-a^3)^2=a(a-a^3)(1-a^2)\in A^{qnil}$ by Lemma 2.1. Therefore $a-a^3\in A^{qnil}$, as required.

$\Longleftarrow$ Set $b=\frac{a^2+a}{2}$ and $c=\frac{a^2-a}{2}$. Then we we check that
$$\begin{array}{c}
b^2-b=\frac{1}{4}(a^4+2a^3-a^2-2a)=\frac{1}{4}(a+2)(a^3-a);\\
c^2-c=\frac{1}{4}(a^4-2a^3-a^2+2a)=\frac{1}{4}(a-2)(a^3-a).
\end{array}$$
Hence $b^2-b,c^2-c\in A^{qnil}$. Since $a$ has g-Drazin inverse, it follows by \cite[Theorem 2.7]{Y} that $a^2$ has g-Drazin inverse.
Clearly, $a^2=b+c$, and so $a^2-a^4=(b+c)-(b+c)^2=(b-b^2)+(c-c^2)-2bc.$
On the other hand, $bc=\frac{a^4-a^2}{4}$, and so $$\frac{3}{2}(a^2-a^4)=(b-b^2)+(c-c^2)\in A^{qnil}.$$ In light of Lemma 2.3, $a^2\in A$ has gs-Drazin inverse.
This completes the proof by Lemma 2.3.\end{proof}

\begin{cor} Let $A$ be a Banach algebra, let $a\in A$. If $a\in A$ has generalized Hirano inverse, then $a^n\in A$ has generalized Hirano inverse for any $n\in {\Bbb N}$.\end{cor}
\begin{proof} In view of Theorem 2.4, $a\in A$ has g-Drazin inverse and $a-a^3\in A^{qnil}$. By virtue of \cite[Theorem 2.7]{Y}, $a^n\in A$ has g-Drazin inverse. It follows by ???? that
$a^n-(a^n)^3=a^n-(a^3)^n=(a-a^3)f(a)\in A^{qnil}$ for some $f(t)\in {\Bbb Z}[t]$. According to Theorem 2.4, $a^n\in A$ has generalized Hirano inverse, as asserted.\end{proof}

\begin{lem} Let $A$ be a Banach algebra, and let $a\in A$. Then $a$ has generalized Hirano inverse if and only if $\frac{a^2+a}{2}$ and $\frac{a^2-a}{2}$ have gs-Drazin inverses.\end{lem}
\begin{proof} $\Longrightarrow$ In light of Lemma 2.3, $a^2$ has gs-Drazin inverse, and so it has g-Drazin inverse. It follows by Lemma 2.2, that $a\in A$ has g-Drazin inverse. Therefore
$b:=\frac{a^2+a}{2}$ has g-Drazin inverses. Furthermore, $$\begin{array}{lll}
b^2-b&=&\frac{1}{4}(a+2)(a^3-a)\\
&\in &A^{qnil}.
\end{array}$$ In light of Lemma Lemma 2.2, $b\in A$ has gs-Drazin inverse. Likewise, $\frac{a^2-a}{2}$ has gs-Drazin inverse, as desired.

$\Longleftarrow$ Set $b=\frac{a^2+a}{2}$ and $c=\frac{a^2-a}{2}$. Then $a^2=b+c$. Since $bc=cb$, as in the proof of Theorem 2.4, $a^2\in A$ has gs-Drazin inverse.
This completes the proof by Theorem 2.4.\end{proof}

We have accumulated all the information necessary to prove the following.

\begin{thm} Let $A$ be a Banach algebra, and let $a\in A$. Then the following are equivalent:\end{thm}
\begin{enumerate}
\item [(1)]{\it $a\in A$ has generalized Hirano inverse.}
\vspace{-.5mm}
\item [(2)]{\it There exists $e^3=e\in comm(a)$ such that $a-e\in A^{qnil}$.}
\end{enumerate}
\begin{proof} $\Longrightarrow$ Let $b=\frac{a^2+a}{2}$ and $c=\frac{a^2-a}{2}$. In view of Lemma 2.6, $b$ and $c$ have gs-Drazin inverses.
Then there exist $f^2=f\in comm^2(b)$ and $g^2=g\in comm^2(c)$ such that $$b-f,c-g\in A^{qnil}.$$
As $ab=ba$ and $ac=ca$, we see that $fa=af$ and $ga=ag$. Hence $gb=bg$. This implies that $fg=gf$.
Therefore $a=b-c=(f-g)+(b-f)-(c-g)$. Clearly, $(b-f)(c-g)=(c-g)(b-f)$.
In light of $(b-f)-(c-g)\in A^{qnil}$. Moreover, we check that $(f-g)^3=f-g$. Set $e=f-g$. Then $a-e\in A^{qnil}$, as required.

$\Longleftarrow$ By hypothesis, there exist $e^3=e\in comm(a)$ such that $w:=a-e\in A^{qnil}$. Hence, $a=e+w$, and so $a^2=e^2+(2e+w)w$. Then $a^2-e^2=(2e+w)w\in A^{qnil}$. In light of ???, $a^2\in A$ has gs-Drazin inverse. Therefore we complete the proof, by ???.\end{proof}

\begin{cor} Let ${\Bbb C}$ be the field of complex numbers, and let $A\in M_n({\Bbb C})$. Then the following are equivalent:\end{cor}
\begin{enumerate}
\item [(1)]{\it $A$ has generalized Hirano inverse.}
\vspace{-.5mm}
\item [(2)]{\it $A$ is the sum of a tripotent and a nilpotent matrices that commutate.}
\vspace{-.5mm}
\item [(3)]{\it The eigenvalues of $A$ are only $-1, 0$ or $1$.}
\vspace{-.5mm}
\item [(4)]{\it $A$ is similar to $diag(J_1,\cdots,J_r)$, where $$J_i=\left(
\begin{array}{ccccc}
\lambda &1&&&\\
&\lambda&&\ddots&\\
&&&\ddots&1\\
&&&&\lambda
\end{array}
\right),\lambda=-1, 0~\mbox{or} ~1.$$}
\end{enumerate}\begin{proof} ????????? In view of Theorem 2.4, $A\in M_n({\Bbb C})$ has Hirano inverse if and only if $A^2\in M_n({\Bbb C})$ has gs-Drazin inverse if and only if the eigenvalues of $A^2$ are only $0$ or $1$. Therefore we are done by~\cite[Lemma 2.2]{W}.\end{proof}

We close with a characterization of a generalized Hirano inverse in terms of its double commutant.

\begin{prop} Let $A$ be a Banach algebra, and let $a\in A$. Then the following are equivalent:\end{prop}
\begin{enumerate}
\item [(1)]{\it $a$ has generalized Hirano inverse.}
\vspace{-.5mm}
\item [(2)]{\it There exists $e^3=e\in comm^2(a)$ such that $a-e\in A^{qnil}$.}
\vspace{-.5mm}
\item [(3)]{\it There exists $b\in comm^2(a)$ such that $$b=(ab)^2, a^2-a^2b\in A^{qnil}.$$}
\end{enumerate}\begin{proof} $(1)\Rightarrow (2)$ In view of Lemma ???, $\frac{a^2+a}{2}$ and $\frac{a^2-a}{2}$ have gs-Drazin inverses. As in the proof of Corollary ???, we can find idempotents $f,g\in comm^2(a)$ such that $a-(f-g)\in A^{qnil}$. Let $e=f-g$. Then $e^3=e\in comm^2(a)$, as required.

$(2)\Rightarrow (3)$ Set $b=(a^2+1-e^2)^{-1}e^2$. As in the proof of Corollary ???, we see that $b=(ab)^2, a^2-a^2b\in A^{qnil}.$ Since $e\in comm^2(a)$, we check that $b\in comm^2(a)$, as desired.

$(3)\Rightarrow (1)$ This is obvious by Corollary ????.\end{proof}

\section{Multiplicative property}

Let $A$ be a Banach algebra, and let $a,b \in A$. In \cite [Lemma 2.2]{L}, it was proved that $ab\in A^{qnil}$ if and only if $ba\in A^{qnil}$. We generalized this fact as follows.

\begin{lem} Let $A$ be a Banach algebra, and let $a,b,c,d\in A$. If $$\begin{array}{c}
(ac)^2a=(db)^2a\\
(ac)^2d=(db)^2d,
\end{array}$$ then then the following are equivalent:\end{lem}
\begin{enumerate}
\item [(1)]{\it $(ac)^2\in A^{qnil}$.}
\vspace{-.5mm}
\item [(2)]{\it $(bd)^2\in A^{qnil}$.}
\end{enumerate}\begin{proof} As $(ac)^2a=(db)^2a$. Then $acaca=dbdba$ which implies  $acacaca=dbdbaca$. Let $aca=a^{'}, c=c^{'}, dbd=d^{'}$ and $b=b^{'}$. Then we have
$a^{'}c^{'}a^{'}=d^{'}b^{'}a^{'}$. Also by $(ac)^2d=(db)^2d$ we have $acacd=dbdbd$ and so    $acacdbd=dbdbdbd$  which imlies $a^{'}c^{'}d^{'}=d^{'}b^{'}d{'}$. Let $(ac)^2\in A^{qnil}$, then $acac\in A^{qnil}$ which implies that $a^{'}c^{'}\in A^{qnil}$. By applying \cite[Lemma 3.1]{??} we conclude that $d^{'}b^{'}\in A^{qnil}$. The converse follows in a similar way.  \end{proof}

\begin{lem} Let $A$ be a Banach algebra, and let $a,b,c,d\in A$. If $$\begin{array}{c}
(ac)^2a=(db)^2a\\
(ac)^2d=(db)^2d,
\end{array}$$ then then the following are equivalent:\end{lem}
\begin{enumerate}
\item [(1)]{\it $(ac)^2\in A^{d}$.}
\vspace{-.5mm}
\item [(2)]{\it $(bd)^2\in A^{d}$.}
\end{enumerate}
\begin{proof} By the same argument in Lemma 3.1, Let $aca=a^{'}, c=c^{'}, dbd=d^{'}$ and $b=b^{'}$. Then we have
$a^{'}c^{'}a^{'}=d^{'}b^{'}a^{'}$ and $a^{'}c^{'}d^{'}=d^{'}b^{'}d^{'}$. Let $(ac)^2\in A^{d}$. This implies that $a^{'}c^{'}a\in A^{d}$. By using \cite[Theorem 3.2]{??} we deduce that $b^{'}d^{'}\in A^{d}$. The converse can be obtained in a similar route.\end{proof}

We come now to the main result of this section.

\begin{thm} Let $A$ be a Banach algebra, and let $a,b,c,d\in A$. If $$\begin{array}{c}
(ac)^2a=(db)^2a\\
(ac)^2d=(db)^2d,
\end{array}$$ then the following are equivalent:\end{thm}
\begin{enumerate}
\item [(1)]{\it $ac\in A$ has Hirano inverse.}
\vspace{-.5mm}
\item [(2)]{\it $bd\in A$ has Hirano inverse.}
\end{enumerate}
\begin{proof} Let $ac\in A$ have Hirano inverse, then by Theorem 2.4, $ac$ has g-Drazin inverse, so by \cite[??]{??} $(ac)^2$ has g-Drazin inverse and Lemma 3.2 implies that $(bd)^2$ has g-Drazin inverse. Also Theorem 2.4 implies that $ac-(ac)^3\in A^{qnil}$.  By Lemma 2.1, $ac(ac-(ac)^3)\in A^{qnil}$  which implies that $(ac)^2-(ac)^4\in A^{qnil}$.  Then we have, $(((1-acac)a)c)^2=((ac)^2-(ac)^4)(1-acac)\in A^{qnil}$.
Let $a^{'}=(1-acac)a, c^{'}=c, b^{'}=b$ and $d^{'}=(1-dbdb)d$. Then $(a^{'}c^{'})^2\in A^{qnil}$. Also
$$\begin{array}{lll}
( a^{'}c^{'})^2a^{'}&=&((1-acac)ac)^2((1-acac)a)\\
&=&(ac)^2-2(ac)^4+(ac)^6)(a-(ac)^2a)\\
&=&(ac)^2a-3(ac)^4a+3(ac)^6a-(ac)^8a\\
&=&(db)^2a-3(db)^4a+3(db)^6a-(db)^8a\\
&=&((1-dbdb)db)^2a^{'}\\
&=&(d^{'}b^{'})^2a^{'}.
\end{array}$$ By the same way we can prove that $( a^{'}c^{'})^2d^{'}=(d^{'}b^{'})^2d^{'}.$ Then by Lemma 3.2, $(b^{'}d^{'})^2\in A^{qnil}$ which implies that $(bd)^2-(bd)^4\in a^{qnil}$ and so
$$(bd-(bd)^3)^2=((bd)^2-(bd)^4)(1-bd)^2\in A^{qnil}$$. Then by ??? $bd$ has ps- Drazin invere and according to ??? $bd$ has generalized Hirano inverse.\end{proof}

\begin{cor} Let $A$ be a Banach algebra, and let $a,b,c,d\in A$. If $$\begin{array}{c}
aca=dba\\
dbd=acd,
\end{array}$$ then the following are equivalent:\end{cor}
\begin{enumerate}
\item [(1)]{\it $(ac)^2\in A^{qnil}$.}
\vspace{-.5mm}
\item [(2)]{\it $(bd)^2\in A^{qnil}$.}
\end{enumerate}\begin{proof} Let $aca=dba$ and $dbd=acd$. Then $(ac)^2a=(db)^2a$ and
$(ac)^2d=(db)^2d$. So the result follows from Lemma 3.1.   \end{proof}

\begin{cor} Let $A$ be a Banach algebra, and let $a,b,c\in A$. If $aba=aca$, then the following are equivalent:\end{cor}
\begin{enumerate}
\item [(1)]{\it $ac\in A$ has Hirano inverse.}
\vspace{-.5mm}
\item [(2)]{\it $ba\in A$ has Hirano inverse.}
\end{enumerate}
\begin{proof} Let $d=a$ It is easy to show that  $(ac)^2a=(db)^2a$ and
$(ac)^2d=(db)^d$. So the result follows from Theorem 3.3. \end{proof}

\begin{cor} Let $A$ be a Banach algebra, and let $a,b,c,d\in A$. If $acac=dbdb$, then the following are equivalent:\end{cor}
\begin{enumerate}
\item [(1)]{\it $ac\in A$ has Hirano inverse.}
\vspace{-.5mm}
\item [(2)]{\it $bd\in A$ has Hirano inverse.}
\end{enumerate}
\begin{proof}  It is easy to show that  $(ac)^2a=(db)^2a$ and
$(ac)^2d=(db)^d$. So the result follows from Theorem 3.3. \end{proof}

In particular, $ab\in A$ has generalized Hirano inverse if and only if $ba\in A$ has generalized Hirano inverse. We note that If $aca=dba,
dbd=acd$ then $(ac)^2a=(db)^2a, (ac)^2d=(db)^2d$. But the converse is not true.

\begin{exam} Let $\sigma$ be an operator, acting on separable Hilbert space $l_2({\Bbb N})$, defined by $$\sigma(x_1,x_2,x_3,x_4,\cdots )=(0,x_1,x_2,0,0,\cdots ),$$ and let $A=M_2(L(l_2({\Bbb N})))$. Choose $$a=
\left(
\begin{array}{cc}
0&\sigma\\
0&0
\end{array}
\right), b=\left(
\begin{array}{cc}
1&0\\
0&0
\end{array}
\right), c=\left(
\begin{array}{cc}
1&0\\
1&1
\end{array}
\right), d=a.$$ Then $(ac)^2a=(db)^2a, (ac)^2d=(db)^2d$, but $aca\neq dba$. In this case, $ac\in A$ has generalized Hirano inverse.\end{exam}

\section{Additive property}

Now we are concern on additive property of
generalized Hirano inverses in a Banach algebra $A$. Since every generalized Hirano invertible element in a Banach algebra has g-Drazin inverse,
we now derive

\begin{lem} Let $A$ be a Banach algebra, and let $a,b\in A$ be Hirano polar. If $a^2b=aba$ and $b^2a=bab$, then $ab$ has generalized Hirano inverse.
\end{lem}
\begin{proof} In view of Theorem ??? and ???, $ab\in A$ has g-Drazin inverse. One easily checks that $ab-(ab)^3=ab-(aba)bab=ab-a^2(b^2a)b=ab-a^2(bab)b=ab-a(aba)b^2=ab-a^3b^3=ab-a^3b^3.$ Set $x=(a-a^3)b$ and $y=a^3(b-b^3)$. Then $ab-(ab)^3=x+y$.

Let $c=a-a^3$.$$\begin{array}{lll}
c^2b&=&(a-a^3)^2b\\
&=&(a^2-2a^4+a^6)b\\
&=&a^2b-2a^4b+a^6b\\
&=&(a-a^3)b(a-a^3)\\
&=&cbc.
\end{array}$$ Likewise, we have $b^2c=bcb$. In light of Theorem ???, $a-a^3\in A^{qnil}$. It follows by Lemma ??? that
$x\in A^{qnil}$. Similarly, $y\in A^{qnil}$.

On the other hand, $$\begin{array}{lll}
x^2y&=&(a-a^3)b(a-a^3)ba^3(b-b^3)\\
&=&(a-a^3)b(a-a^3)ba^3(1-b^2)b\\
&=&aba^3-aba^3b^2-a^3ba^3+a^3ba^3b^2\\
&=&a^3ba-a^3ba^3-a^3b^3a+a^3b^3a^3\\
&=&(a-a^3)ba^3(b-b^3)(a-a^3)b\\
&=&xyx.
\end{array}$$ Likewise, $y^2x=yxy$. By using Lemma ??? again, $x+y\in A^{qnil}$. Therefore $ab-(ab)^3=x+y\in A^{qnil}$, by ????. This completes the proof, by ?????.\end{proof}

\begin{lem} Let $A$ be a Banach algebra, let $a,b\in A$ be Hirano polar and $ab=ba$. If $1+a^db$ has generalized Hirano inverse, then $a+b$ has generalized Hirano inverse.
\end{lem}
\begin{proof} In view of ???, $a+b\in A$ has g-Drazin inverse.
Since $1+a^db$ has generalized Hirano inverse, we see that $1+a^db-(1+a^db)^3=a^db-(a^db)^3-3a^db(1+a^db)\in A^{qnil}$. In light of Lemma ???, $a^db\in A$ has generalized Hirano inverse, and so
$a^db-(a^db)^3\in A^{qnil}$. In view of Lemma ???, we see that $3a^db(1+a^db)\in A^{qnil}$, and so
$$\begin{array}{lll}
3ab(a+b)&=&3(a-a^2a^d)b(a+b)+3a^2b(aa^d+a^db)\\
&=&3(a-a^2a^d)b(a+b)+3a^2b(aa^d+a(a^d)^2b)\\
&=&3(a-a^2a^d)b(a+b)+3a^3a^db(1+a^db)\\
&\in &A^{qnil}.
\end{array}$$
Consequently, $(a+b)-(a+b)^3=(a-a^3)+(b-b^3)-3ab(a+b)\in A^{qnil}$.
Accordingly, $a+b$ has generalized Hirano inverse, by Theorem ????.\end{proof}

Let $p\in A$ be an idempotent, and
let $x\in A$. Then we write $$x=pxp+px(1-p)+(1-p)xp+(1-p)x(1-p),$$
and induce a representation given by the matrix
$$x=\left(\begin{array}{cc}
pxp&px(1-p)\\
(1-p)xp&(1-p)x(1-p)
\end{array}
\right)_p,$$ and so we may regard such matrix as an element in
$A$. For any idempotent $e$ in $A$, $(eAe)^{qnil}\subseteq A^{qnil}$.

We now ready to prove the following.

\begin{thm} Let $A$ be a Banach algebra, and let $a,b\in A$ be Hirano polar. If $a^2b=aba$ and $b^2a=bab$, then $a+b$ has generalized Hirano inverse if and only if $1+a^db$ has generalized Hirano inverse.
\end{thm}
\begin{proof} $\Longrightarrow$ Write
$1+a^db=x+y$ where $x=(1-aa^d)$ and $y=a^d(a+b)$. Then $x^2=x\in A$ has generalized Hirano inverse and $xy=0$.
One easily checks that $(a^d)^2(a+b)=a^d(a+b)a^d$,
$(a+b)^2a^d=(a+b)a^d(a+b)$. In light of Lemma ???, $a^d(a+b)\in A$ has generalized Hirano inverse. Therefore
$1+a^db=x+y\in A$ is Hirao polar, by Lemma ?????.

$\Longleftarrow$ Step 1. We check that $1+(a^2a^d)^d(aa^d)b=1+(aaa^d)^d(aa^db)=1+a^daa^daa^db=1+a(a^d)^2b=1+a^db\in A$ has generalized Hirano inverse.
Since $(a^2a^d)(aa^db)=(aa^db)(a^2a^d)$, it follows by Lemma ??? that
$a^2a^d+aa^db=aa^d(a+b)\in A$ has generalized Hirano inverse.

Step 2. $b\in A^{qnil}$. Then $(1-aa^d)(a+b)=x+y$ where
$x=(a-a^2a^d)$ and $y=(1-aa^d)b$. Then $x^2y=xyx$ and $y^2x=yxy$.
Clearly, $x=a-a^2a^d, y=(1-aa^d)b\in A^{qnil}$. Then $(1-aa^d)(a+b)=x+y\in A^{qnil}$. Choose
$p=aa^d$. Then $$a+b=\left(
\begin{array}{cc}
p(a+b)p&0\\
(1-p)(a+b)p&(1-p)(a+b)(1-p)\end{array} \right)_p.$$ Since $a^d\in comm(ab)$, we see that $(1-p)(a+b)(1-p)=(1-p)(a+b)\in A^{qnil}$.
Clearly, $(pa)(pb)=(pb)(pa)$ and $(pb)^d=0$. It follows by Lemma ??? that $p(a+b)p=p(a+b)=pa+pb\in A$ has generalized Hirano inverse.
In light of Lemma ???, $a+b\in A$ has generalized Hirano inverse.

Step 3. Choose $p=bb^d$. Then $a=\left(
\begin{array}{cc}
a_1&0\\
*&a_2\end{array} \right)_p$ and $b=\left(
\begin{array}{cc}
b_1&0\\
*&b_2\end{array} \right)_p$ where $a_1=pap$, $a_2=(1-p)a(1-p)$,
$b_1=pbp$ and $b_2=(1-p)b(1-p)$. Hence, $$a+b=\left(
\begin{array}{cc}
a_1+b_1&0\\
*&a_2+b_2\end{array} \right)_p.$$ Thus, $a_1,b_1 \in A$ has generalized Hirano inverse,
$a_1^2b_1=a_1b_1a_1$, $b_1^2a_1=b_1a_1b_1$, $a_1b_1=b_1a_1$ and
$1+a_1^d b_1 \in A$ has generalized Hirano inverse. Therefore, $a_1+b_1\in A$ has generalized Hirano inverse by Lemma
?????. Moreover, $a_2 \in A$ has generalized Hirano inverse, $b_2 \in A^{qnil}$,
$a_2^2b_2=a_2b_2a_2$ and $b_2^2a_2=b_2a_2b_2$. By Step 2, $a_2+b_2
\in A$ has generalized Hirano inverse. Therefore $a+b\in A$ has generalized Hirano inverse by Lemma ???.
\end{proof}

\begin{cor} Let $A$ be a Banach algebra, and let $a,b\in A$ be Hirano polar. If $ab=ba$, then $a+b$ has generalized Hirano inverse if and only if $1+a^db$ has generalized Hirano inverse.
\end{cor}
\begin{proof} This is obvious, by Theorem ???.\end{proof}

For further use, we now record the following.

\begin{prop} Let $A$ be a Banach algebra, and let $a,b\in A$. If $a^2+ab, b^2+ab$ have gs-Drazin inverses and $a^2b+ab^2=0$, then $a+b$ has generalized Hirano inverse.\end{prop}
\begin{proof} Let $$M=\left(
\begin{array}{c}
a\\
1
\end{array}
\right)\left(
\begin{array}{cc}
1&b
\end{array}
\right).$$ Then $$M^2=\left(
\begin{array}{cc}
a^2+ab&0\\
a+b&b^2+ab
\end{array}
\right).$$ Since $a^2+ab, b^2+ab$ have gs-Drazin inverses, we see that $$\begin{array}{lll}
M^2-M^4&=&\left(
\begin{array}{cc}
(a^2+ab)-(a^2+ab)^2&0\\
*&(b^2+ab)-(b^2+ab)^2
\end{array}
\right)\\
&\in& M_2(A)^{qnil}.
\end{array}$$ Thus, $M-M^3\in M_2(A)^{qnil}.$ In light of ????, $$a+b=\left(
\begin{array}{cc}
1&b
\end{array}
\right)\left(
\begin{array}{c}
a\\
1
\end{array}
\right)$$ has generalized Hirano inverse, as asserted.\end{proof}

\begin{cor} Let $A$ be a Banach algebra, and let $a,b\in A$. If $a,b$ are Hirano polar and $ab=0$, then $a+b$ has generalized Hirano inverse.\end{cor}
\begin{proof} Since $a,b$ are Hirano polar, it follows by ??? that $a^2,b^2\in A$ have gs-Drazin inverses. This completes the proof by Theorem ????.\end{proof}

\section{Splitting approach}

We are now concerned on the generalized Hirano inverse for a operator matrix
$M$. Here, \begin{equation} M=\left(
\begin{array}{cc}
A&B\\
C&D
\end{array}
\right)\end{equation}
 where $A,D\in L(X)$
has generalized Hirano inverses and $X$ is a complex Banach space. Then $M$
is a bounded linear operator on $X\oplus X$. Using different splitting of the operator matrix $M$ as $P+Q$, we will apply preceding results to obtain various conditions
for the generalized Hirano inverse of $M$.

\begin{lem} \label{lem5.1} Let $A,D\in L(X)$ have generalized Hirano inverses and $B\in L(X)$. Then $\left(
\begin{array}{cc}
A&B\\
0&D
\end{array}
\right)\in M_2(L(X))$ has generalized Hirano inverse.\end{lem}
\begin{proof} In view of Theorem ???, $A,D\in L(X)$ have g-Drazin inverse and $A-A^3,D-D^3\in A^{qnil}$. In view of ????, $\left(
\begin{array}{cc}
A&B\\
0&D
\end{array}
\right)\in M_2(L(X))$ has g-Hirano inverse. Moreover, it follows by ??? that
$$\left(
\begin{array}{cc}
A&B\\
0&D
\end{array}
\right)-\left(
\begin{array}{cc}
A&B\\
0&D
\end{array}
\right)^3=\left(
\begin{array}{cc}
A-A^3&*\\
0&D-D^3
\end{array}
\right)\in M_2(L(X))^{qnil}.$$  According to Theorem ???, we obtain the result.\end{proof}

\begin{lem} Let $A$ be a Banach algebra, and let $a\in A$ have generalized Hirano inverse. If $e^2=e\in comm(a)$, then $ea\in A$ has generalized Hirano inverse.
\end{lem}
\begin{proof} In light of ???, $ea\in A$ has g-Drazin inverse. Moreover, $a-a^3\in A^{qnil}$, and so $ea-(ea)^3=e(a-a^3)\in A^{qnil}$, by Lemma ???. This completes the proof.\end{proof}

\begin{thm} Let $A,D\in L(X)$ have generalized Hirano inverse and $M$ be given by $(5.1)$. If $BC=CB=0$, $CA(I-A^{\pi})=D^{\pi}DC$ and $A^{\pi}AB=BD(I-D^{\pi})$, then $M\in M_2(L(X))$ has generalized Hirano inverse.
\end{thm}
\begin{proof} Let $$P=\left(
\begin{array}{cc}
A(1-A^{\pi})&B\\
0&DD^{\pi}
\end{array}
\right),~Q=\left(
\begin{array}{cc}
AA^{\pi}&0\\
C&D(I-D^{\pi})
\end{array}
\right).$$ Then $M=P+Q$. Since $A(1-A^{\pi}), DD^{\pi}\in L(X)^{qnil}$, we see that $P\in L(X)^{qnil}$, and so it has generalized Hirano inverse.
On the other hand, $D(I-D^{\pi})=D(DD^d)$. It follows by Lemma ??? that $D(I-D^{\pi})\in L(X)$ has generalized Hirano inverse. Thus
$Q$ has generalized Hirano inverse. It is easy
to verify that
$$PQ=\left(
\begin{array}{cc}
0&BD(I-D^{\pi})\\
DD^{\pi}C&0
\end{array}
\right)=\left(
\begin{array}{cc}
0&AA^{\pi}B\\
CA(I-A^{\pi})&0
\end{array}
\right)=QP.$$ Also we have $$P^d=\left(
\begin{array}{cc}
(A(I-A^{\pi}))^d&X\\
0&D^dD^{\pi}
\end{array}
\right)=\left(
\begin{array}{cc}
A^d&X\\
0&0
\end{array}
\right)$$ where $X=(A^d)^2\sum_{n=0}^{\infty}(A^d)^nB(DD^\pi)^n$.
Hence, $$P^dQ=\left(
\begin{array}{cc}
A^d&X\\
0&0
\end{array}
\right)\left(
\begin{array}{cc}
AA^{\pi}&0\\
C& D(I-D^\pi)
\end{array}
\right)=\left(
\begin{array}{cc}
XC&XD(I-D^{\pi})\\
0&0
\end{array}
\right)$$ where
$XC=(A^d)^2(B+\sum_{n=1}^{\infty}(A^d)^nB(DD^{\pi})^n)C=0$ as
$BC=0$, $B(DD^{\pi})^nC=0$. Moreover, we have
$$\begin{array}{lll}
XD(I-D^{\pi})&=&(A^d)^2(B+\sum_{n=1}^{\infty}(A^d)^nB(DD^{\pi})^n)D(I-D^{\pi})\\
&=&(A^d)^2(B+BD(I-D^{\pi}))\\
&=&(A^d)^2(B+A^{\pi}AB)\\
&=&(A^d)^2B
\end{array}$$ and so $P^dQ=\left(
\begin{array}{cc}
0&(A^d)^2B\\
0&0
\end{array}
\right)$. Thus, $I_2+P^dQ$ is invertible. So, it has g-Drazin
inverse. Further, we have
$$\begin{array}{lll}
(I_2+P^dQ)-(I_2+P^dQ)^3&=&-P^dQ(I_2+P^dQ)(2I_2+P^dQ)\\
&=&\left(
\begin{array}{cc}
0&-(A^d)^2B\\
0&0
\end{array}
\right)\left(
\begin{array}{cc}
2I&3(A^d)^2B\\
0&2I
\end{array}
\right)\\
&=&\left(
\begin{array}{cc}
0&-2(A^d)^2B\\
0&0
\end{array}
\right)\\
&=&\in (M_2(L(X))^{qnil}.
\end{array}$$ In light of Theorem ?????, $I_2+P^dQ\in M_2(L(X))$ has generalized Hirano inverse. Therefore,
we complete the proof by ?????.\end{proof}

In the proof of Theorem ???, we choose $$P=\left(
\begin{array}{cc}
A(1-A^{\pi})&B\\
0&D^2D^{d}
\end{array}
\right),~Q=\left(
\begin{array}{cc}
AA^{\pi}&0\\
C&DD^{\pi}
\end{array}
\right).$$ Analogously, we can derive

\begin{prop} Let $A,D\in L(X)$ have generalized Hirano inverse and $M$ be given by $(5.1)$. If $BC=CB=0$, $CA(I-A^{\pi})=(I-D^{\pi})DC$ and $A^{\pi}AB=BDD^{\pi}$, then $M\in M_2(L(X))$ has generalized Hirano inverse.
\end{prop}

We now turn to the operator matrix $M$ with trivial generalized Schur complement, i.e., $D=CA^dB$ (see ???? ). We have

\begin{thm} Let $A\in L(X)$ has generalized Hirano inverse, $D\in L(X)$ and $M$ be given by $(5.1)$. Let $W=AA^d+A^dBCA^d$. If $AW$ has generalized Hirano inverse,
$$A^{\pi}BC=BCA^{\pi}=AA^{\pi}B=0, D=CA^dB,$$ then $M$ has generalized Hirano inverse.\end{thm}
\begin{proof} We easily see that $$M=\left(
\begin{array}{cc}
A&B\\
C&CA^dB
\end{array}
\right)=P+Q,$$ where $$P=\left(
\begin{array}{cc}
A&AA^dB\\
C&CA^dB
\end{array}
\right),Q=\left(
\begin{array}{cc}
0&A^{\pi}B\\
0&0
\end{array}
\right).$$ By hypothesis, we verify that $QP=0$. Clearly, $Q$ is nilpotent, and so it has
generalized Hirano inverse. Furthermore, we have
$$P=P_1+P_2,~P_1=\left(
\begin{array}{cc}
A^2A^d&AA^dB\\
CAA^d&CA^dB
\end{array}
\right),~P_2=\left(
\begin{array}{cc}
AA^{\pi}&0\\
CA^{\pi}&0
\end{array}
\right)$$ and $P_2P_1=0.$ By virtue of ???????, $P_2$
has generalized Hirano inverse. Obviously, we have $$P_1= \left(
\begin{array}{c}
AA^d\\
CA^d
\end{array}
\right)\left(
\begin{array}{cc}
A&AA^dB
\end{array}
\right).$$ By hypothesis, we see that $$\left(
\begin{array}{cc}
A&AA^dB
\end{array}
\right)\left(
\begin{array}{c}
AA^d\\
CA^d
\end{array}
\right)=AW$$ has generalized Hirano inverse. In light of ??????, $P_1$ has generalized Hirano inverse. Thus, $P$ generalized Hirano inverse, by Lemma ????.
According to Lemma ?????, $M$ has generalized Hirano inverse.\end{proof}


\vskip10mm\begin{thebibliography}{99} \bibitem{CS} M.S. Abodlyousefi and H. Chen, Generalized Hirano inverses in rings, arXiv:1707.09389v1 [math.RA] 28 Jul 2017.

\bibitem{CS2} M.S. Abodlyousefi and H. Chen, Rings in which elements are sums of tripotents and nilpotents, {\it J. Algebra Appl.}, {\bf 17}(2018), 1850042 (11 pages).

\bibitem{C1} D.S. Cvetkovic-Ilic; D.S. Djordjevic and Y. Wei, Additive results for the generalized Drazin inverse in a Banach algebra, {\it Linear Algebra Appl.},
{\bf 418}(2016), 53--61.

\bibitem{D} D.S. Djordjevic and Y. Wei, Additive results for the generalized Drazin inverse, {\it J. Austral. Math. Soc.},
{\bf 73}(2002), 115--125.

\bibitem{G} O. Gurgun, Properties of generalized strongly Drazin invertible
elements in general rings, {\it J. Algebra Appl.}, 16, 1750207 (2017) [13 pages], DOI: 10.1142/S0219498817502073.

\bibitem{H} R. Harte, Invertibility and Singularity for Bounded Linear Operators, Marcel Dekker, New York, 1988.

\bibitem{Y} Y. Jiang. Y. Wen and Q. Zeng, Generalizations of Cline's formula for three generalized inverses, Revista. Un. Math. Argentina, {\bf 48}(2017), 127-134.
 
\bibitem{K} J.J. Koliha, A generalized Drazin inverse, {\it Glasgow Math. J.}, {\bf 38}(1996), 367--381.

\bibitem{L} Y. Liao, J. Chen and J. Cui, Cline's formula for the generalized Drazin inverse, {\it Bull. Malays. Math. Sci. Soc.}, {\bf 37}(2014), 37--42.

\bibitem{Mi} V.G. Miller and H. Zguitti, New extensions of Jacobson's lemma and Cline's formula, {\it Rend. Circ. Mat. Palermo, II. Ser.}, Published online: 09 February 2017, Doi: 10.1007/s12215-017-0298-6.

\bibitem{M} D. Mosic, Extensions of Jacobson's lemma for Drazin inverses, {\it Aequat. Math.}, Published online: 04 April 2017, Doi: 10.1007/s00010-017-0476-9.

\bibitem{Z1} Q. Zeng and H. Zhong, Common properties of bounded linear operators $AC$ and $BA$: local spetral theory, {\it J. Math. Anal. Appl.},
{\bf 414}(2014), 553--560.

\bibitem{Z2} Q. Zeng and H. Zhong, New results on common properties of the products $AC$ and $BA$, {\it J. Math. Anal. Appl.},
{\bf 427}(2015), 830--840.

\bibitem{Z} Q. Zeng, Z. Wu and Y. Wen, New extensions of Cline's formula for generalized inverses, {\it Filomat},
{\bf 31}(2017), 1973--1980.


\end{thebibliography}
\end{document}